\documentclass{siamart251216}

\usepackage[english]{babel}
\usepackage[T1]{fontenc}
\usepackage[utf8]{inputenc}
\usepackage{csquotes}

\usepackage{amsmath}
\usepackage{algorithm, algpseudocode}

\usepackage{graphicx}
\usepackage{mathtools}
\mathtoolsset{showonlyrefs,showmanualtags}
\usepackage{xfrac}
\usepackage{mathrsfs}
\usepackage{dsfont}
\usepackage{xcolor}

\usepackage{booktabs}
\usepackage{siunitx}

\allowdisplaybreaks

\usepackage{xspace}

\newsiamremark{example}{Example}
\newsiamremark{remark}{Remark}
\crefname{hypothesis}{Hypothesis}{Hypotheses}
\newsiamthm{fact}{Fact}
\crefname{fact}{Fact}{Facts}

\DeclareMathOperator{\dist}{dist} 

\DeclareMathOperator{\closu}{cl} 
\DeclareMathOperator{\sign}{sign}

\def\RR{\mathds R}
\def\NN{\mathds N}

\def\Mcal{\mathcal M}
\def\Bcal{\mathcal B}
\def\Rcal{\mathcal R}
\def\Bbar{\overline{\Bcal}}
\def\Scal{\mathcal{S}}
\def\supportS{\mathds{S}}

\def\Scal{\mathcal{S}}

\def\BPMAP{\texttt{BPMap}\xspace}

\def\BPMAPhoc{\texttt{BPMap-HOC}\xspace}
\def\BPMAPbin{\texttt{BPMap-BIN}\xspace}
\def\BPhocbin{\texttt{BPMap-HOC-BIN}\xspace}

\DeclareMathOperator{\bap}{bap}

\newcommand{\NormTwo}[1]{\left\lVert #1 \right\rVert_2}
\newcommand{\NormOne}[1]{\left\lVert #1 \right\rVert_1}
\newcommand{\NormInf}[1]{\left\lVert #1 \right\rVert_\infty}

\newcommand{\abs}[1]{\left\lvert #1 \right\rvert}
\newcommand{\scal}[2]{\left\langle #1, #2 \right\rangle}

\usepackage[shortlabels]{enumitem}
\newlist{lista}{enumerate}{1}
\setlist[lista]{label=\alph*., nosep,leftmargin=*,align=right}

\newlist{listi}{enumerate}{1}
\setlist[listi]{label={\upshape(\roman*\upshape)},leftmargin=*,align=right, widest=iii,nosep, format=\bf}

\begin{document}

\title{Basis pursuit by inconsistent alternating projections\thanks{
\funding{\textbf{RB} was partially supported by Brazilian agency Fundação de Amparo à Pesquisa do Estado do
    Rio de Janeiro (Grant E-26/201.345/2021) and Conselho Nacional de    Desenvolvimento Científico e Tecnológico -- CNPq (Grants 407147/2023-3 and 403197/2025-2); \textbf{YBC} was partially supported by the USA agency National Science    Foundation (Grant DMS-2307328);
     \textbf{LRS} was also partially funded by Brazilian agencies CNPq (Grants 310571/2023-5, 407147/2023-3 and 403197/2025-2) and Fundação de Amparo  ao Estado de Santa Catarina -- FAPESC (Edital 21/2024, Grant 2024TR002238)}; \textbf{PJSS} was partially supported by CNPq (Grant 312394/2023-3) and FAPESP (Grants 2013/07375-0, 2023/08706-1).    
     }}

\author{Roger Behling\thanks{Departamento de Matemática, Universidade Federal de Santa Catarina, Blumenau, SC, Brazil 
(\email{rogerbehling@gmail.com}, \email{l.r.santos@ufsc.br}).}
\and Yunier Bello-Cruz\thanks{Department of Mathematical Sciences, Northern Illinois University, DeKalb, IL, USA 
(\email{yunierbello@niu.edu}).}
\and Luiz-Rafael Santos\footnotemark[1]
\and \break Paulo J.~S.~Silva\thanks{Department of Applied Mathematics, IMECC/Universidade Estadual de Campinas, Campinas, SP, Brazil 
(\email{pjssilva@unicamp.br}).}}

\maketitle

\begin{abstract}
  Basis pursuit is the problem of finding a vector with smallest $\ell_1$-norm among the solutions of a given linear system of equations. It is a well-known convex relaxation of the sparse affine feasibility problem, where sparse solutions to underdetermined systems are sought. Since basis pursuit admits a linear programming reformulation, standard LP solvers are directly applicable.  
  We instead address the basis pursuit directly in its $\ell_1$-minimization form, without LP reformulation, via a scheme that uses alternating projections in its subproblems. These subproblems are designed to be inconsistent in the sense that they relate to two non-intersecting sets. Recently in [R.~Behling, Y.~Bello-Cruz and L.-R.~Santos, \emph{SIAM J. Optim.}, 31 (2021), pp.~2863–2892], inconsistency coming from infeasibility has been shown to accelerate convergence of alternating projections. We deliberately enforce this inconsistency by constructing subproblems whose feasible sets are disjoint by design. We prove that the resulting $\ell_1$-radii converge linearly to the optimal
value, and that when the solution is unique, all generated sequences converge
linearly to it at a rate governed by a natural error bound between the
feasible set and the optimal $\ell_1$-ball. The proposed method is numerically
  competitive against state-of-the-art open-source solvers on synthetic and
  real-world instances.

\end{abstract}


\section{Introduction}\label{sec:intro}
Recovering an unknown signal $x^\ast \in \RR^n$ from linear measurements is a central problem
in signal processing and related fields. The measurements take the form:
\begin{equation}\label{IP}
    A x = b,
\end{equation}
where \( A \in \RR^{m \times n} \) is a known matrix with full row rank, and \( b = A x^\ast \in \RR^m \) represents the observed data. In typical scenarios where $m\le n$, the linear system \eqref{IP} is underdetermined and thus has infinitely many solutions.

To address this ill-posedness, it is often assumed that the original signal $x^\ast$ is sparse. A widely used approach for recovering such signals is to solve the \emph{Basis Pursuit Problem} \cite{Chen:2001} 
\begin{equation}\label{BP}
\min \| x \|_1 \qquad \mbox{s.t.} \qquad Ax=b. \tag{BP}
      \end{equation}

Here and throughout, $\|\cdot\|_1$  denotes the $\ell_1$-norm, \emph{i.e.}, for any $x\in\RR^n$ $\|x\|_1\coloneqq\sum_{i=1}^n|x_i|$. The \(\ell_1\)-norm serves as a convex proxy for sparsity, promoting solutions with few nonzero entries. A key assumption, both standard and desirable, is that problem \eqref{BP} admits a unique solution, which ideally coincides with the true sparse vector $x^\ast$. 
This uniqueness plays a critical role in theoretical recovery guarantees and has been the focus of extensive research; see, for instance, \cite{Bruckstein:2009,Candes:2005,Gilbert:2017,Mousavi:2019,Bello-Cruz:2022,Demanet:2016,Zhang:2015} for a partial list of relevant contributions. 

The problem \eqref{BP} is a convex relaxation of the \emph{Sparse Affine Feasibility Problem}, which is the problem of finding a sparse vector $x$ that satisfies $Ax=b$, that is,
\begin{equation}\label{L0}
\min \, \lVert x \rVert_0 \qquad \text{s.t.} \qquad Ax=b, \tag{SFP}
\end{equation} where $\lVert x \rVert_0$ counts the number of nonzeros entries of $x$. For an $x\in \RR^n$, we define $$\supportS = \supportS(x) \coloneqq \{ i\in \{1,\ldots,n\}\mid x_i\neq 0\}$$ as the \emph{support} of $x$.  In this sense, $\lVert x \rVert_0 = |\supportS(x)|$. Also, using $\supportS$, we denote $A_{\supportS}$ as the submatrix of $A$ formed by the columns indexed by $\supportS$ whereas $x_{\supportS}$ is the subvector of $x$ built from the entries indexed by $\supportS$.

Problem \eqref{L0} is a fundamental problem in Compressed Sensing (CS),  the framework for recovering sparse signals from incomplete measurements; see \cite{Bruckstein:2009,Donoho:2006}. While \eqref{L0} is known to be NP-hard, there are some instances where the solutions of \eqref{BP} and \eqref{L0} coincide \cite{Candes:2006,Donoho:2006}.  For instance, it was established in \cite{Candes:2008a}  that the \emph{Restricted Isometry Property} (RIP) of matrix $A$ suffices for the correspondence of solutions to  \eqref{L0} with solutions to \eqref{BP}. Unfortunately,  RIP is a strong condition that is not easy to verify, since it is NP-hard~\cite{Tillmann:2014}. Another useful condition ensuring equivalence between \eqref{BP} and \eqref{L0} is the \emph{Exact Recover Condition} (ERC) \cite[Thm.~3.1]{Tropp:2004}, which depends on the support $\supportS$ of a solution; of course, finding such a set is a combinatorial problem~\cite{Dessole:2023}.

The importance of the Basis Pursuit Problem \eqref{BP} led to a great deal of research devoted to the development of efficient methods for solving it, mainly for large-scale problems. Several approaches have been proposed for solving \eqref{BP}. One can think, for instance, of reformulating it as a linear program (LP) and applying to the reformulation  any standard LP solver~\cite{Candes:2005}. The use of splitting methods is also a popular strategy. For instance, in \cite{Hesse:2014,Demanet:2016}, the  Douglas-Rachford method was employed to solve \eqref{BP}. Given the nonsmoothness of the $\ell_1$-norm, subgradient-based methods were also applied to solve the Basis Pursuit Problem \eqref{BP} \cite{Lorenz:2014}. Lorenz and co-authors \cite{Lorenz:2015} not only compare many solvers applied to \eqref{BP}, but also establish a useful heuristic for an optimality certificate.; see more in Section~\ref{sec:numerical} where we compare the implementation of our method with some of these methods.

In this paper, we propose a new technique for solving \eqref{BP} based on the method of Projection onto Convex Sets (POCS)~\cite{Combettes:1990}. POCS is a classical method for solving convex feasibility problems, that is, problems of the form 
\begin{equation}\label{eq:feasibility}
\text{find } w \in \bigcap_{i=1}^m W_i,
\tag{CFP}
\end{equation}
where $W_i$ are closed convex sets in $\RR^n$ employing orthogonal projections. The point $P_{W}(w)\in W$ is the projection of  $w\in \RR^n$ onto the closed  convex set $W\subset\RR^n$ if and only if $\scal{y-P_{W}(w)}{w - P_{W}(x)}\le 0,$ for all $y\in W$ where $\scal{\cdot}{\cdot}$ denotes the Euclidean inner product in $\RR^n$. Hence, $P_W(w)=\arg\min_{y\in W}\|y-w\|_2$ where $\NormTwo{\cdot}$ refers to the $\ell_2$-norm, \emph{i.e.}, for any $x\in\RR^n$, $\NormTwo{x}\coloneqq\sqrt{\scal{x}{x}}$. POCS basically employs successive alternating projections onto the considered sets $W_i$. Famously,  when only two sets are considered, POCS is known as the \emph{Method of Alternating Projections} (MAP) 
\cite{Bauschke:1993,Bauschke:1996}. 
Indeed, given any point $w\in \RR^{n}$ and supposing $W,Y\subset \RR^{n}$ are  closed, convex and nonempty sets, the sequence $(w^{k})_{k\in\NN}$ given by
\begin{equation}\label{eq:MAPsequence}
w^0 = w,\quad w^{k+1} = P_WP_Y(w^{k}),
\end{equation}
for every $k\in\NN$ is called  the \emph{Alternating Projection Sequence} starting at $w^{0}=w$.

The solution of problem \eqref{BP} can be characterized as a two-set convex feasibility problem, where the underlying sets are the affine subspace $\Mcal \coloneqq  \{x \in \RR^n \mid Ax = b\}$ and the (optimal) $\ell_1$-ball  $\Bbar \coloneqq  \{x \in \RR^n \mid \NormOne{x} \leq \bar r\}$, where $\bar r \geq 0$ is the optimal value of \eqref{BP}.  This means that the solution set of problem  \eqref{BP}, denoted throughout the text as $\Scal$, is precisely $ \Mcal\cap \Bbar$. 

Naturally, one does not know $\bar r$ in advance. To circumvent this,  our proposed algorithm called BP-MAP computes a sequence $(r_k)_{k\in\NN}$ such that $r_k \uparrow \bar r$, as $k \to \infty$.  While $r_k < \bar r$, we have infeasibility, that is,  $\Mcal\cap \Bcal_1(0,r_k) = \emptyset$. 
The main idea behind our method is to use MAP to seek a \emph{Best Approximation Pair} (BAP) between  $\Mcal$ and $\Bcal^k = \Bcal_1(0,r_k) \coloneqq \{z \in \RR^n\mid \NormOne{z} \leq  r_k \}$, \emph{i.e.}, a pair $(x^k,z^k) \in \Mcal \times  \Bcal_k$ that minimizes the Euclidean distance between these two sets. Once found, we use this distance to define $r_{k+1}$ and thus the sequence $(r_k)_{k\in\NN}$ is generated. Recently, in \cite{Behling:2021}, it was shown that infeasibility can be beneficial for the convergence of MAP, even providing possible finite convergence, depending on some regularity conditions on the underlying sets. In this work, we adapt some ideas from \cite{Behling:2021} into the context of the Basis Pursuit Problem \eqref{BP}.

In summary, we initialize BP-MAP by setting $r_0 \coloneqq 0$ and $z^0 \coloneqq 0 \in \RR^n$. Then,  for all $k\geq 0$, we compute $x^k \coloneqq P_{\Mcal}(z^k)$, updating $r_{k+1} = r_k + \NormTwo{d^k}$, where $d^k \coloneqq z^k - x^k$, and then we compute $z^{k+1}$ by inner MAP iterations that seek a BAP regarding  $\Mcal$ and $\Bcal^k$. Algorithm~\ref{alg:BP_MAP} formally presents the method.  

\begin{algorithm}[H]\small
  \begin{algorithmic}[1]
  \Require {$r_0 \coloneqq 0 \in \RR$, $z^0 \coloneqq 0 \in \RR^n$ and $\Bcal^0 \coloneqq \{z^0\}$} 
  \For{$k= 0,1,2, \ldots$}
     \State Set
      \begin{equation}\label{eq:def_x_k}
        x^k \coloneqq P_{\Mcal}(z^k)
      \end{equation}\label{alg:BP_MAP_stepProj}
      \State Compute the displacement vector $d^k \coloneqq z^k - x^k$
      \State Update 
      \begin{equation}\label{eq:def_r_k}
      r_{k+1} = r_k + \NormTwo{d^k}
      \end{equation}
      \label{alg:BP_MAP_rk_update}
      \State Find $z^{k+1}$ using MAP between $\Bcal^{k+1} \coloneqq   \Bcal_1(0;r_{k+1})$ and $\Mcal$, that is,  
    set $z^{k,0} \coloneqq z^k$ and compute  
   \begin{equation}\label{eq:def_z_k+1}
      z^{k+1} = \lim_{j\to \infty} P_{\Bcal^{k+1}} P_{\Mcal} (z^{k,j})      \end{equation}

      \label{alg:BP_MAP_stepMAP}
%
\EndFor
\end{algorithmic}
  \caption{BP-MAP -- Infeasible MAP for Basis Pursuit}\label{alg:BP_MAP}  
\end{algorithm}

The paper is organized as follows. Section~\ref{sec:facts} collects background on MAP needed
for the analysis of Algorithm~\ref{alg:BP_MAP}. Section~\ref{sec:convergence} establishes
convergence of BP-MAP, including linear convergence rates. Section~\ref{sec:numerical} reports
numerical experiments comparing Algorithm~\ref{alg:BP_MAP} with two state-of-the-art methods,
and introduces two variants aimed at improving practical performance. Section~\ref{sec:concluding}
closes with final remarks and directions for future work.

\section{Well-definedness of Algorithm~\ref{alg:BP_MAP}}\label{sec:facts}

In this section we collect some background material on MAP that leads to the well-definedness of Algorithm~\ref{alg:BP_MAP}.

Let  $W,Y\subset \RR^{n}$ be closed, convex and nonempty. 
The \emph{distance} between $W$ and $Y$ is given by
$\label{eq:dist}
\dist(W,Y) \coloneqq \inf \{\NormTwo{w-y}\mid w\in W, y\in Y\}.$
We call a  pair $(\bar x,\bar y)\in W\times Y$ attaining  the distance between $W$ and $Y$, that is,  $\dist(\bar w,\bar y) = \dist(W,Y)$, as \emph{Best Approximation Pair} (BAP)  relative to $W$ and $Y$.  The set of all BAPs relative to $W$ and $Y$ is denoted by $\bap(W,Y)\subset W\times Y $ and, accordingly, we define the  sets 
\[\bap_{Y}(W) \coloneqq\{x\in W\mid (x,y)\in \bap(W,Y)\}
\] and 
\[\bap_{W}(Y) \coloneqq\{y\in Y\mid (x,y)\in \bap(W,Y)\}
.
\] 
Note that $\bap_{Y}(W)$ (respectively $\bap_{W}(Y)$) is the subset of points in $W$  (respectively, in $Y$) \emph{nearest} to $Y$ (respectively, to $W$).

In the consistent case, that is,  when $W\cap Y$ is nonempty, we have $\bap_{W}(Y) = \bap_{Y}(W) = W\cap Y$. When $W\cap Y = \emptyset$, we define the \emph{displacement vector} as $d\coloneqq P_{\closu(W-Y)}(0)$. Hence, $\dist(W,Y) = \NormTwo{d}$ and $\dist(W,Y)$ is attained if and only if $d\in W-Y$. 

Cheney and Goldstein~\cite{Cheney:1959} established that if one of the sets is compact or if one of the sets is finitely generated, the  fixed point set of the operator $P_WP_Y(\cdot)$ is nonempty and the sequence \eqref{eq:MAPsequence} converges to a fixed point of this operator. The general result was latter summarized and extended by  \cite[Thm.~4.8]{Bauschke:1994} as follows. 


\begin{proposition}[Convergence of MAP]\label{prop:conv-MAP} Let $(w^k)_{k\in\NN}$ be an Alternating Projection Sequence as given in \eqref{eq:MAPsequence}. Then,  \[w^k-P_Y(w^k)\to d,\] where $d$ is called the \emph{displacement vector} regarding $W$ and $Y$.  Moreover, 
if $\dist(W, Y)$ is attained then   $w^k\to \bar w\in \bap_{Y}(W)$ and $P_Y(w^k)\to \bar y\coloneqq \bar w-d\in \bap_{W}(Y)$;
\end{proposition}



We now establish the well-definedness of Algorithm~\ref{alg:BP_MAP}.

\begin{theorem}\label{thm:well-definedness}
  Algorithm~\ref{alg:BP_MAP} is well-defined. In particular, the limit
  in~\eqref{eq:def_z_k+1} exists and is finite for every $k \geq 0$.
\end{theorem}

\begin{proof}
Steps~\ref{alg:BP_MAP_stepProj}--\ref{alg:BP_MAP_rk_update} present no difficulty: $P_{\Mcal}(z^k)$ is well-defined because $\Mcal$ is a closed nonempty affine subspace of $\RR^n$; $d^k = z^k - x^k$ is then a finite vector; and $r_{k+1} = r_k + \NormTwo{d^k}$ is a finite real number.

It remains to verify Step~\ref{alg:BP_MAP_stepMAP}, that is, that the limit in~\eqref{eq:def_z_k+1} exists. Both $\Mcal$ and $\Bcal^k =\Bcal_1(0;r_k)$ are polyhedra: $\Mcal$ is defined by the linear system $Ax = b$, and $\Bcal_1(0;r_k)$ is the convex hull of the $2n$ vertices $\{\pm r_k e_i\}_{i=1}^n$, where $e_i$ are the standard basis vectors. Since the distance between two polyhedra is always
attained~\cite[Thm.~5]{Cheney:1959}, Proposition~\ref{prop:conv-MAP}
applies and guarantees that the alternating projection sequence
$(z^{k,j})_{j\in\NN}$ converges to a point in $\bap_{\Mcal}(\Bcal^k)$.
Hence, the limit in~\eqref{eq:def_z_k+1} exists for every $k \geq 0$.
\end{proof}

As a consequence, $d^k = z^k - x^k$ is the displacement vector between
$\Bcal^k$ and $\Mcal$, that is, $\dist(\Bcal^k, \Mcal) = \NormTwo{d^k}$
for every $k \geq 0$.

\section{Convergence analysis}\label{sec:convergence}

In this section we prove that the sequence generated by Algorithm~\ref{alg:BP_MAP} is bounded and all its cluster points are solutions of the Basis Pursuit Problem \eqref{BP}. As an immediate consequence, if \eqref{BP} has a unique solution, the sequence converges to that solution.
Regarding the rate of convergence, we prove that the approach to the optimal ball regarding \eqref{BP} is linear, with respect to the $\ell_1$-radii generated by Algorithm~\ref{alg:BP_MAP}.

We begin by establishing that the radii in Algorithm~\ref{alg:BP_MAP} are monotonically increasing and converging to the optimal value of \eqref{BP}, which is the corresponding radius of the optimal $\ell_1$-ball.
\begin{lemma} \label{lemma:r_k_to_rbar}
  Let $\bar r$ be the optimal value of~\eqref{BP} and let the sequence
  $(r_k)_{k\in\NN}$ be defined by~\eqref{eq:def_r_k} in
  Algorithm~\ref{alg:BP_MAP}. Then,
  \begin{listi}
    \item for all $k \geq 0$, we have $0 \leq r_k \leq r_{k+1} \leq \bar{r}$;
    \item the sequence $(r_k)_{k\in\NN}$ converges to $\bar r$;
    \item the sequence $(d^k)_{k\in\NN}$ converges to 0.
  \end{listi}
\end{lemma}

\begin{proof}
  If $\bar r = 0$, then $0\in \RR^n$ is the unique solution of \eqref{BP} and the claim is true, because Algorithm~\ref{alg:BP_MAP} generates constant sequences such that $z^k = x^k = 0$, for every $k\geq 0$. 
  
 Let us suppose $\bar r >0$. Regarding item (i), since $r_{k+1}\coloneqq r_k+ \NormTwo{d^k} = r_k+\NormTwo{z^k-x^k}$  and $r_0\coloneqq 0$, we obtain $0 \leq r_k \leq r_{k+1}$, for all $k\geq 0$. 
  
  We now prove by induction that $r_{k} \leq \bar r$, for all $k \geq 0$. First, note that $r_0 = 0 < \bar r$. Now consider the induction hypothesis $r_{k} \leq \bar r$, for some $k \geq 0$. Let $\bar{x}$ be a solution of \eqref{BP}. Then, $\NormOne{\bar x} = \bar r$ and  \( \dfrac{r_k}{\bar{r}} \bar{x} \in \Bcal^k  = \Bcal_1(0;r_k).\) 
Indeed,
\[
\NormOne{\frac{r_k}{\bar{r}} \bar{x}}=\frac{r_k}{\bar{r}}\NormOne{\bar{x}}=\frac{r_k}{\bar{r}} \cdot \bar{r}=r_k .
\]
Thus, 
\begin{equation}\label{eq:def_r_k+1}
r_{k+1} = r_k + \NormTwo{z^k-x^k} \leq r_k+\NormTwo{\frac{r_k}{\bar{r}} \bar{x} - \bar{x}}\end{equation}
because $\left(z^k, x^k\right) \in \bap\left(\Bcal^k, \Mcal \right)$, $\dfrac{r_k}{\bar{r}} \bar{x}\in \Bcal^k$, and $\bar{x} \in \Mcal$.
Therefore, 
\[
\begin{aligned} r_{k+1} &\leq r_k+\NormTwo{\dfrac{r_k}{\bar{r}} \bar{x} - \bar{x}} \leq r_k+\NormOne{\dfrac{r_k}{\bar{r}} \bar{x} - \bar{x}} 
\\
& =r_k+\abs{\frac{r_k}{\bar{r}} - 1}\|\bar{x}\|_1=r_k+\abs{\frac{r_k}{\bar{r}} - 1} \cdot \bar{r} \\
& =r_k+\left( 1- \frac{r_k}{\bar{r}}\right) \cdot \bar{r}=r_k+\bar{r}-r_k=\bar{r} ,
\end{aligned}
\]
where the first inequality is by \eqref{eq:def_r_k+1}, the second follows from  norm equivalence~\cite[\S 2.2.2]{Golub:2013}, and we used the induction hypothesis to get $\frac{r_k}{\bar r} \leq  1$. Hence, item (i) holds.

As for item (ii) and (iii), the monotonicity and boundedness established in item (i) give $r_k \uparrow \hat r \leq \bar r$. Since $\Mcal$ is closed, the map $r \mapsto \dist(\Bcal_1(0;r), \Mcal)$ is continuous, so
  $\dist(\Bcal^k, \Mcal) \to \dist(\Bcal_1(0;\hat r), \Mcal)$. On the other hand, $\dist(\Bcal^k, \Mcal) = \NormTwo{d^k} = r_{k+1} - r_k \to 0$.  So, not only does $d^k \to 0$ but  $ \dist(\Bcal_1(0;\hat r), \Mcal)  =0$,  which implies that $\hat r = \bar r$. This establishes (ii) and (iii).
 \end{proof}

As a consequence of Lemma~\ref{lemma:r_k_to_rbar}, we have the following.

\begin{lemma}\label{lemma:NormOneConvergence}
  The scalar sequences $(\NormOne{z^k})_{k\in\NN}$ and
  $(\NormOne{x^k})_{k\in\NN}$ arising from Algorithm~\ref{alg:BP_MAP}
  converge to the optimal value $\bar r$ of~\eqref{BP}.
\end{lemma}

\begin{proof}
  We first show that $\NormOne{z^k} = r_k$ for all $k \geq 0$. Since
  $(x^k, z^k)$ is a best approximation pair between the disjoint sets
  $\Mcal$ and $\Bcal_1(0;r_k)$, the point $z^k$ must lie on the boundary
  of $\Bcal_1(0;r_k)$: if it were in its interior, one could move $z^k$ toward $x^k$
  and strictly decrease the distance to $\Mcal$, contradicting optimality;
  see~\eqref{eq:def_x_k} and~\eqref{eq:def_z_k+1}. Hence, $\NormOne{z^k} =
  r_k$, and Lemma~\ref{lemma:r_k_to_rbar}(ii) gives $\NormOne{z^k} \to
  \bar r$.

  For $(\NormOne{x^k})_{k\in\NN}$, note that by the triangle inequality,
  norm equivalence, and~\eqref{eq:def_r_k},
  \begin{equation}\label{eq:NormOneradius_ineq}
    \abs{\NormOne{z^k} - \NormOne{x^k}}
    \leq \NormOne{z^k - x^k}
    \leq \sqrt{n}\,\NormTwo{z^k - x^k}
    = \sqrt{n}\,\NormTwo{d^k}
    = \sqrt{n}(r_{k+1} - r_k).
  \end{equation}
  Since $\NormOne{z^k} \to \bar r$ and $r_{k+1} - r_k = \NormTwo{d^k} \to 0$
  by Lemma~\ref{lemma:r_k_to_rbar}(iii), we conclude that $\NormOne{x^k}
  \to \bar r$.
\end{proof}

The following is the main convergence result of this section.

\begin{theorem}\label{thm:convergence}
  The sequences $(z^k)_{k\in\NN}$ and $(x^k)_{k\in\NN}$ generated by
  Algorithm~\ref{alg:BP_MAP} are bounded, and all their cluster points are
  solutions of~\eqref{BP}.
\end{theorem}

\begin{proof}
  Since $\NormOne{z^k} = r_k \leq \bar r$ for all $k \geq 0$ by
  Lemma~\ref{lemma:NormOneConvergence} and Lemma~\ref{lemma:r_k_to_rbar}(i),
  norm equivalence in $\RR^n$ gives that $(z^k)_{k\in\NN}$ is bounded.
  Because $x^k = P_{\Mcal}(z^k)$ and projections are nonexpansive,
  $(x^k)_{k\in\NN}$ is bounded as well.

  Let $\bar x \in \RR^n$ be a cluster point of $(x^k)_{k\in\NN}$, arising
  from a subsequence $(x^{k_j})_{j\in\NN}$. Since $x^k \in \Mcal$ for all
  $k \geq 0$ and $\Mcal$ is closed, $\bar x \in \Mcal$. By
  Lemma~\ref{lemma:NormOneConvergence}, $\NormOne{x^{k_j}} \to \bar r$, so
  $\NormOne{\bar x} = \bar r$, which gives $\bar x \in \Mcal \cap \Bbar =
  \Scal$. Thus, $\bar x$ is a solution of~\eqref{BP}.

  Finally, since $d^k \to 0$ by Lemma~\ref{lemma:r_k_to_rbar}(iii) and
  $z^k = x^k + d^k$, the corresponding subsequence satisfies $z^{k_j} =
  x^{k_j} + d^{k_j} \to \bar x$. Therefore, every cluster point of
  $(z^k)_{k\in\NN}$ is also a solution of~\eqref{BP}.
\end{proof}

From now on, we assume that \eqref{BP} has a unique solution $x^\ast\in\RR^n$. In most practical  problems, \eqref{BP} features unique solutions. As we mentioned in the introduction, the uniqueness of the solution of \eqref{BP} is both standard and desirable, and has been extensively studied. With that in mind, it follows from the previous theorem  that both sequences generated by Algorithm~\ref{alg:BP_MAP} fully converge. 

\begin{corollary}\label{cor:convergence}
  Assume that \eqref{BP} has a unique solution $x^\ast \in \RR^n$. Then, the sequences
  $(z^k)_{k\in\NN}$ and $(x^k)_{k\in\NN}$ generated by Algorithm~\ref{alg:BP_MAP} converge
  to $x^\ast$.
\end{corollary}
\begin{proof}
  Immediate from Theorem~\ref{thm:convergence}: a bounded sequence, for which every cluster point equals $x^\ast$,  converges to $x^\ast$.
\end{proof}

We now turn to investigating convergence rates related to Algorithm~\ref{alg:BP_MAP}. For that, note that our algorithm is built upon geometrical characteristics involving the affine set $\Mcal$ and the optimal ball $\Bbar$. In particular, the way they touch provides an error bound condition as presented below.


\begin{proposition}\label{prop:ErrorBound}
  There exists $\omega \in (0,1]$ such that for all $z \in \Bbar$, we have
  \begin{equation}\label{eq:omega}
    \omega \dist(z, \Scal) \leq \dist(z, \Mcal).
  \end{equation}
\end{proposition}
\begin{proof}
  Since $\Bbar$ is a polyhedron and its intersection $\Scal = \Mcal \cap \Bbar$
  with the affine subspace $\Mcal$ is nonempty, Hoffman's
  lemma~\cite{Hoffman:1952} provides a constant $\kappa > 0$ such that
  $\dist(z, \Scal) \leq \kappa\,\dist(z, \Mcal)$ for all $z \in \Bbar$.
  Setting $\omega \coloneqq 1/\kappa$ yields~\eqref{eq:omega} with $\omega > 0$.
  The bound $\omega \leq 1$ holds because $\Scal \subset \Mcal$ implies
  $\dist(z, \Mcal) \leq \dist(z, \Scal)$ for all $z$.
\end{proof}

This error bound condition is at the core of the linear convergence results stated in sequel.

\begin{theorem}[Linear convergence of $(r_k)_{k\in\NN}$ and $({d^k})_{k\in\NN}$ ]
  \label{thm:linear_convergence_rk} 
  Let $(r_k)_{k\in\NN}$ and $({d^k})_{k\in\NN}$ be the  sequences generated by Algorithm~\ref{alg:BP_MAP} and $\omega\in (0,1]$ as in Proposition~\ref{prop:ErrorBound}. Then, 

  \begin{listi}
    \item $(r_k)_{k\in\NN}$ converges linearly to $\bar r$ with rate no worse than  $1-\omega/\sqrt{n}$;
    \item $({d^k})_{k\in\NN}$ converges linearly to $0$ with rate no worse than  $1-\omega/\sqrt{n}$.
  \end{listi}

\end{theorem}
\begin{proof}
Let $\bar z^k \coloneqq P_{\Scal}(z^k)$. By Lemma~\ref{lemma:NormOneConvergence},
$\NormOne{z^k} = r_k \leq \bar r$, so $z^k \in \Bbar$ and
Proposition~\ref{prop:ErrorBound} gives
\begin{equation}\label{eq:dist_z^k_S_EB}
  \omega\,\NormTwo{z^k - \bar z^k} \leq \dist(z^k, \Mcal) = \NormTwo{d^k}.
\end{equation}
By norm equivalence and the reverse triangle inequality,
\[
  \NormTwo{z^k - \bar z^k}
  \geq \frac{\NormOne{z^k - \bar z^k}}{\sqrt{n}}
  \geq \frac{\abs{\NormOne{z^k} - \NormOne{\bar z^k}}}{\sqrt{n}}
  = \frac{\bar r - r_k}{\sqrt{n}},
\]
where the last equality uses $\NormOne{z^k} = r_k$ and $\NormOne{\bar z^k} =
\bar r$ (since $\bar z^k \in \Scal \subset \partial\Bbar$). Combining
with~\eqref{eq:dist_z^k_S_EB},
\[
  \NormTwo{d^k} \geq \frac{\omega(\bar r - r_k)}{\sqrt{n}}.
\]
Since $r_{k+1} - r_k = \NormTwo{d^k}$ by~\eqref{eq:def_r_k}, we get
\begin{equation}\label{eq:linear_convergence_rk}
  \bar r - r_{k+1}
  = (\bar r - r_k) - \NormTwo{d^k}
  \leq \left(1 - \frac{\omega}{\sqrt{n}}\right)(\bar r - r_k),
\end{equation}
which is item~(i). For item~(ii), since $r_{k+1} \leq \bar r$ by item~(i), we have
\[
  \NormTwo{d^k} = r_{k+1} - r_k \leq \bar r - r_k.
\]
Applying~\eqref{eq:linear_convergence_rk} inductively gives
$\bar r - r_k \leq \rho^k(\bar r - r_0) = \rho^k \bar r$,
and therefore
\[
  \NormTwo{d^k} \leq \rho^k \bar r,
\]
which tends to $0$ linearly with rate $\rho$. Therefore, the theorem holds.
\end{proof}

\begin{theorem}\label{thm:linear_convergence}
  Assume that~\eqref{BP} has a unique solution $x^\ast \in \RR^n$, and let
  $\rho \coloneqq 1 - \omega/\sqrt{n}$, with $\omega \in (0,1]$ as in
  Proposition~\ref{prop:ErrorBound}. Then, the sequences $(z^k)_{k\in\NN}$
  and $(x^k)_{k\in\NN}$ generated by Algorithm~\ref{alg:BP_MAP} converge
  $R$-linearly to $x^\ast$ with rate no worse than $\rho$. More precisely, for
  all $k \geq 0$,
  \begin{equation}\label{eq:linear_rate_zk_xk}
    \NormTwo{z^k - x^\ast} \leq \frac{\bar r}{\omega}\,\rho^k
    \qquad\text{and}\qquad
    \NormTwo{x^k - x^\ast} \leq \frac{\bar r}{\omega}\,\rho^k.
  \end{equation}
\end{theorem}

\begin{proof}
  We first argue that $z^k \in \Bbar$ for all $k \geq 0$. Since
  $(z^k, x^k) \in \bap(\Bcal^k, \Mcal)$, the point $z^k$ must lie on the
  boundary of $\Bcal^k$: if it were in its interior, one could move $z^k$ toward $x^k$
  and strictly decrease the distance to $\Mcal$, contradicting optimality.
  Thus, $\NormOne{z^k} = r_k$, and Lemma~\ref{lemma:r_k_to_rbar}(i) gives
  $r_k \leq \bar r$, so $z^k \in \Bbar$.

  Since $\Scal = \{x^\ast\}$, we can apply Proposition~\ref{prop:ErrorBound}
  to $z^k \in \Bbar$ to get
  \begin{equation}\label{eq:EB_unique}
    \omega\,\NormTwo{z^k - x^\ast}
    = \omega\,\dist(z^k, \Scal)
    \leq \dist(z^k, \Mcal)
    = \NormTwo{z^k - x^k}
    = \NormTwo{d^k},
  \end{equation}
  where the last two equalities use that $x^k = P_{\Mcal}(z^k)$ is the
  nearest point in $\Mcal$ to $z^k$; see~\eqref{eq:def_x_k}. Dividing through by $\omega > 0$ and applying
  Theorem~\ref{thm:linear_convergence_rk}(ii), which gives
  $\NormTwo{d^k} \leq \rho^k \bar r$, we obtain
  \[
    \NormTwo{z^k - x^\ast} \leq \frac{\NormTwo{d^k}}{\omega}
    \leq \frac{\bar r}{\omega}\,\rho^k.
  \]
  The bound for $(x^k)_{k\in\NN}$ follows immediately: since $x^\ast \in \Mcal$,
  we have $x^\ast = P_{\Mcal}(x^\ast)$, and nonexpansiveness of $P_{\Mcal}$
  gives
  \[
    \NormTwo{x^k - x^\ast}
    = \NormTwo{P_{\Mcal}(z^k) - P_{\Mcal}(x^\ast)}
    \leq \NormTwo{z^k - x^\ast}
    \leq \frac{\bar r}{\omega}\,\rho^k.
  \]
\end{proof}

To summarize: the radii $(r_k)_{k\in\NN}$ converge linearly to the optimal
value $\bar r$, and both $(z^k)_{k\in\NN}$ and $(x^k)_{k\in\NN}$ converge
linearly to the unique solution $x^\ast$ when~\eqref{BP} has a unique
solution. In the non-unique case, all cluster points are solutions
of~\eqref{BP}, and full convergence remains open.
 
Unlike a naive scheme that simply moves from one best approximation pair to
the next, BP-MAP uses MAP for its inner iterations, which carries Fejér
monotonicity properties; see~\cite[Chap.~5]{Bauschke:2017a}. We note
that none of the convergence results above used any property of MAP beyond
convergence to best pairs — so any inner method seeking best pairs could replace MAP in Step~\ref{alg:BP_MAP_stepMAP} with the same theoretical
guarantees. However, leaping from one best pair to another without Fejér structure is not enough for convergence; simple counterexamples exist in~$\RR^3$. Our conjecture is that $(x^k)_{k\in\NN}$ is quasi-Fejér in the sense of~\cite{Combettes:2001a,Behling:2024d}, and hence convergent regardless of uniqueness. A proof is left for future work.


\section{Numerical experiments}\label{sec:numerical}

In this section, we present numerical results for Algorithm~\ref{alg:BP_MAP}. It was implemented as a package in the \texttt{Julia} language~\cite{Bezanson:2017}; the code is fully available at \url{https://github.com/lrsantos11/BP_MAP}. We denote this implementation as \BPMAP. In this implementation, the projection onto the affine space $\Mcal$ is performed using the {\tt ProximalOperators.jl} package~\cite{Stella:2025} if the data matrix $A$ is dense and the Conjugate Gradient implementation from {\tt Krylov.jl}~\cite{Montoison:2023} if $A$ is sparse. For the projection onto the $\ell_1$-ball, we used an implementation of the Newton semismooth method described in~\cite{Cominetti:2014}, with modifications to incorporate ideas from~\cite{Condat:2016} in the initial iteration and to leverage multiple threads~\cite{Secchin2026}.

We compare our implementation with two state-of-the-art open-source solvers that implement the best approaches for \eqref{BP}, as described in~\cite{Lorenz:2015}. In this paper, the authors test several solvers for \eqref{BP}, and the best results are obtained using a dual Simplex implementation for a Linear Programming (LP) model representing \eqref{BP} and an infeasible-point subgradient algorithm from~\cite{Lorenz:2014}. We limit ourselves to open-source implementations to avoid comparing against very specialized codes that depend on ample resources available in commercial development.

For the dual simplex approach, we use the \texttt{HiGHS} solver, a well-regarded and high-performance code for LP~\cite{Huangfu:2018a}. It is used to solve a natural LP problem that models \eqref{BP}~\cite{Candes:2005}. Indeed, the following model is equivalent to \eqref{BP}:
\begin{equation}
\label{eq:LP}
\min_{x^+, x^-} \mathds{1}^{\top} x^{+}+\mathds{1}^{\top} x^{-} \text { s.t. } A x^{+}-A x^{-}=b,\quad x^{+} \geq 0,\ x^{-} \geq 0.
\end{equation}
Here, $A \in \RR^{m \times n}$ and the vector $b \in \RR^{m}$ are the problem data and $\mathds{1} \in \RR^n$ is a vector of ones. The uses the standard split of free variables $x = x^{+}-x^{-}$, where $x^{+}\coloneqq\max \{0, x\}$ and $x^{-}\coloneqq\max \{0,-x\}$ due to minimization. \texttt{HiGHS} is well suited to solve \eqref{eq:LP}. The solver is written in C++ and has easy bindings in \texttt{Julia}. The solution of \eqref{BP} using HIGHS to solve~\eqref{eq:LP} will be referred to as \texttt{HiGHS-BP}. 

The second solver is called \texttt{ISAL1}. It is based on the infeasible-point subgradient algorithm proposed in~\cite{Lorenz:2014} for general convex optimization problems that was adapted in \cite{Lorenz:2015} to solve \eqref{BP}. An iteration of \texttt{ISAL1} is given by
\begin{equation}\label{eq:ISAL1}
x^{k+1}\coloneqq P_\Mcal\left(x^k-\lambda_k \frac{\NormOne{x^k}-\varphi}{\NormTwo{h^k}^2} h^k\right)
\end{equation}
where ${P}_\Mcal$ is the projection onto the affine manifold $\Mcal$, $\left(\lambda_k\right)$ are vanishing positive step-size parameters, $\varphi$ is a lower bound on the optimal objective function value, and $h^k \in \partial\left\|x^k\right\|_1$ is a subgradient of the $\ell_1$-norm at the current point $x^k$. \texttt{ISAL1} is implemented in \texttt{Matlab} and is freely available. We use a \texttt{Matlab} interface to call this solver from \texttt{Julia}.

We consider two different test sets. The first is a collection of synthetic \eqref{BP} instances described in \cite{Lorenz:2013}. We use these instances for an initial assessment of our algorithm's performance and to establish variants that accelerate its numerical efficiency. After that, we compare the best method to the two solvers mentioned above. The second set of instances is adapted from real-world LASSO~\cite{Tibshirani:1996} instances compiled in~\cite{Lopes:2019}. Again, we benchmark the best \BPMAP variant and the two other solvers in this set of tests.

\subsection{Variants of \BPMAP}\label{sec:variants}

Before presenting the numerical results, we outline two variants of \BPMAP (Algorithm~\ref{alg:BP_MAP}) that aim to improve its performance.

\subsubsection*{\BPMAPhoc}

Lorenz, Pftesch and Tillmann~\cite[Sec. 2]{Lorenz:2015} proposed an optimality check based on a heuristic that tries to compute an optimal point based on the sign structure of the current iterate. The rationale is that most algorithms produce an infinite sequence converging to a solution; however, due to the geometry of \eqref{BP}, the correct sign structure may be identified early. Using this routine within a solver, once the correct sign structure is achieved, it can recover a solution. For completeness, we present the heuristic below.

\begin{algorithm}[H]\small
  \begin{algorithmic}[1]
  \Require {matrix $A$, right hand side vector $b \neq 0$, vector $x$} 
  \State Deduce candidate for support set  $\supportS \coloneqq \supportS(x)$ of $x$.
  \State compute approximate solution $\hat w$ to $A_\supportS^\top w = \sign(x_\supportS)$
  \If{$\NormInf{A\top \hat w} \approx 1$ }
    \If{$\hat{x}$ exists with $A_\supportS \hat{x}_\supportS=b$ and $\hat{x}_i=0,  \forall i \notin \supportS$}  
      \If{ $\left(\|\hat{x}\|_1+b^{\top}(-\hat{w})\right) /\|\hat{x}\|_1 \approx 0$}
        \State \Return \texttt{success} ($\hat{x}$ is an optimal solution)
      \EndIf
    \EndIf
  \EndIf
  \State \Return \texttt{failure} (no optimal solution found)
\end{algorithmic}
  \caption{{Heuristic Optimality Check (HOC)~\cite[Alg.~2]{Lorenz:2015}}}\label{alg:HOC}  
\end{algorithm}

As explained by \cite{Lorenz:2015}, HOC can be employed as part of several $\ell_1$ solvers, if the solvers approach a solution asymptotically. This is the case of \texttt{ISAL1}, and the version we are comparing to our implementation. It uses HOC. The results of ISAL1 with HOC are much better than those without it. 

In our case, we adapted HOC for use within Algorithm~\ref{alg:BP_MAP}. The idea is to call HOC when the support set of the current iterate is the same as the previous one. This indicates that the sign structure is stabilizing. We call this variant \BPMAPhoc. We observe that the HOC routine is called after Step \ref{alg:BP_MAP_stepProj} of Algorithm~\ref{alg:BP_MAP}. The two linear systems present in the HOC description are solved using Julia's backslash operator, if $A$ is dense, and the {\tt cgls} function from {\tt Krylov.jl}, if $A$ is sparse. 

Before presenting the first numerical results, we describe some implementation
details below. 
\begin{enumerate}
  \item All code was run in a computing node with two AMD EPYC 9255 24-Core CPUs, 1.5TB of RAM, using Debian GNU/Linux 6.1. The maximum number of threads available was 24; hence, at most one of the CPUs was employed in the tests. 
  \item When calling MAP for $\Bcal^k \cap \Mcal$, the method assumes that a best approximation pair was found as a proof of infeasibility if the Euclidean distance between two successive projection pairs does not show a relative improvement of at least $\num{e-6}$.
  \item When calling MAP for $\Bcal^k \cap \Mcal$, an approximate point in the intersection is considered found whenever the infinite distance between a projection onto $\Bcal^k$ and the projection onto $\Mcal$ is smaller than $\num{e-6}$ in relative or absolute values.
  \item In the \BPMAP method, an approximate solution is declared found whenever the MAP call returns that it found an approximate point in the intersection of $\Bcal^k \cap \Mcal$.
  \item The {\tt cg} function from {\tt Krylov.jl} is called with its default options when used to project onto $\Mcal$. It uses the square root of the machine epsilon as a precision threshold. As we are using double precision, this value is approximately $\num{1.49e-8}$.
  \item In HOC, when {\tt cgls} is used to solve the linear systems, it is also called with the default options. Hence, it uses $\num{1.49e-8}$ as both relative and absolute precision thresholds.
  \item In line 3 of HOC, the equalities are checked using absolute precision and $\num{e-6}$ as threshold.
  \item All solvers had a limit of one hour to solve a problem. After that, it should declare failure.
\end{enumerate}

Figure~\ref{fig:bpmapvsbphoc} shows the performance profile~\cite{Dolan:2002} comparing these two variants in the Lorenz; see a detailed description below. It is clear from this experiment that using HOC greatly accelerates the convergence of \BPMAP and, hence, it will be adopted as the default from now on.
\begin{figure}[htbp]
  \centering
  \includegraphics[width=0.75\textwidth]{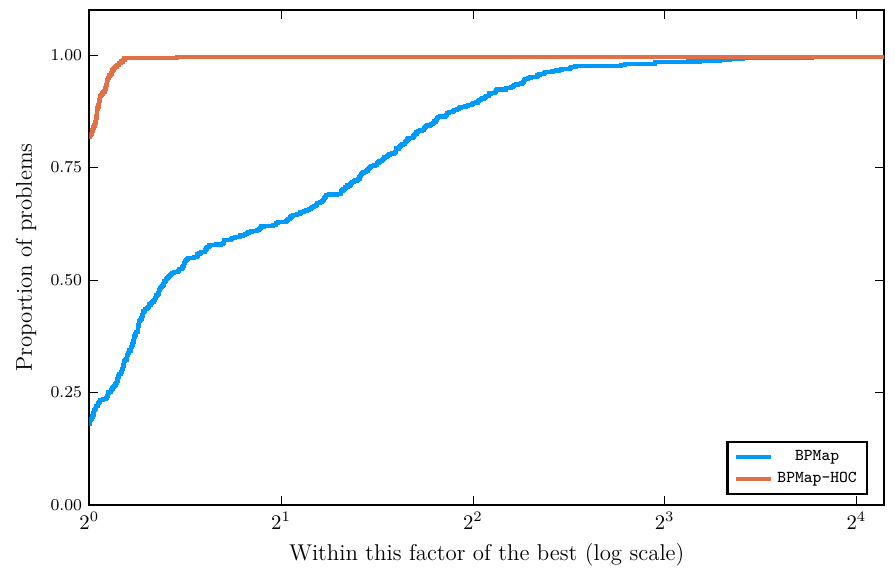} 
  \caption{Performance profile comparing \BPMAP and \BPMAPhoc in the Lorenz test set.}
  \label{fig:bpmapvsbphoc}
\end{figure}

\subsubsection*{\BPMAPbin}

The second variant we propose is a binary search to update the radius $r_k$ in Algorithm~\ref{alg:BP_MAP}. This Algorithm computes $\NormOne{ z^k} = r_k \leq \bar{r} \leq \NormOne{x^k}$, with $r_k$ and $\NormTwo{x^k}$ converging linearly to the radius of the optimal $\ell_1$-ball $\bar{r}$. The fact that $r_k$ is always a lower estimate of the correct radius is important for the algorithm's efficiency, as the MAP algorithm will be applied to an infeasible problem, where MAP typically converges fast. However, the linear convergence rate depends on the problem data and can be slow.

An alternative approach to approximating $\bar{r}$ with a guaranteed linear rate is to use binary search, as described in Algorithm~\ref{alg:BPMAPbin}. Clearly, $R_k - r_k \rightarrow 0$ linearly with rate $\min(\alpha, 1 - \alpha)$. The same holds for the convergence of $r_k$ and $R_k$ to the optimal radius $\bar{r}$.

\begin{algorithm}[H]\small
  \begin{algorithmic}[1]
  \Require $r_0 \coloneqq 0 \leq R_0 \coloneqq \| P_\Mcal (0) \|_1$, $\alpha \in (0, 1)$.
  \For{$k= 0,1,2, \ldots$}
     \State Define $\Rcal^\alpha_k \coloneq \alpha r_k + (1 - \alpha) R_k$.  
     \State Use MAP do determine if $\Mcal \cap \Bcal_1(0, \Rcal^\alpha_k)$ is empty.
     \If{$\Mcal \cap \Bcal_1(0, \Rcal^\alpha_k) = \emptyset$}
       \State $r_{k + 1} \coloneq \Rcal^\alpha_k$.
     \Else
       \State $R_{k + 1} \coloneq \Rcal^\alpha_k$.
     \EndIf
  \EndFor 
  \end{algorithmic}
  \caption{\BPMAPbin\ -- Binary search of \BPMAP\ for Basis Pursuit}
  \label{alg:BPMAPbin}  
\end{algorithm}

In practice, the implementation of \BPMAPbin is just a slight variant of \BPMAP. It can also profit from using the HOC heuristic from the previous section. We will refer to the HOC variant of \BPMAPbin as \BPhocbin. A fundamental parameter in the algorithm is the value of $\alpha$. While we performed the first tests, it was clear that the natural value of $\frac{1}{2}$ is not ideal. With this $\alpha$, in many cases, $\Mcal \cap \Bcal_1(0, \Rcal^\alpha_k) \neq \emptyset$ and the MAP algorithm would take too long to find a point in the intersection. We then settled on using $\alpha = 0.9$, favoring empty intersections. Moreover, the implementation would declare success whenever the absolute or relative difference between $r_k$ and $R_k$ drops below $\num{e-6}$.

Figure~\ref{fig:bphocvsbphocbin} shows the performance profile of the HOC variants of \BPMAP and \BPMAPbin. The binary variant is faster. Hence, we will use this version in the comparison with the other methods below. 

\begin{figure}[htbp]
\centering
\includegraphics[width=0.75\textwidth]{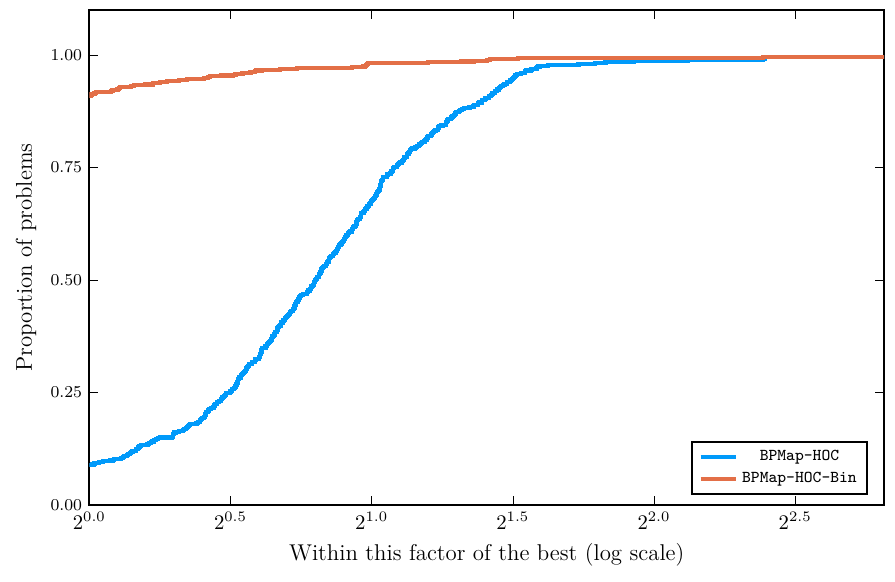} 
\caption{Performance profile comparing \BPMAPhoc and \BPhocbin in the Lorenz test set.}
\label{fig:bphocvsbphocbin}
\end{figure}

\subsection{Comparison to other methods}

\subsubsection*{Synthetic instances -- The Lorenz test set}
\label{sec:synthetic}

The first part of our experiments considers a set of synthetic instances specifically designed to test algorithms for Basis Pursuit. These instances were introduced in~\cite{Lorenz:2013} and used in~\cite{Lorenz:2015} to conduct an extensive comparison of several state-of-the-art solvers for \eqref{BP}. They are fully available at \url{http://wwwopt.mathematik.tu-darmstadt.de/spear/software/L1_Comparison/SPEAR_L1_Testset_mat.zip}; see \cite[Sec.~4]{Lorenz:2015} for further details.

The test set contains a total of \num{548} instances, generated from \num{100} explicitly defined matrices \( A \in \RR^{m \times n} \). It covers a diverse range of constructions, including random matrices (Gaussian, binary, and ternary) and structured transforms (Hadamard, Haar, sinusoidal, and convolution), as well as their concatenations. The matrices include both sparse and dense examples: sparse matrices are those with \( m \geq \num{2048} \), while matrices with \( m \leq \num{1024} \) are dense. All matrix columns are normalized to have unit Euclidean norm and ensured to be unique.

Each matrix is paired with \num{4} to \num{6} right-hand side vectors \( b \coloneqq Ax \), generated from sparse vectors \( x \) with either high or low dynamic range. For \num{400} of the instances, the support of \( x \) satisfies the Exact Recovery Condition (ERC) \cite{Tropp:2004}, guaranteeing uniqueness of the solution to \eqref{BP}. The remaining \num{148} instances were constructed using dual certificate techniques to ensure uniqueness, even if the ERC fails.

The problem sizes vary with \( m \in \{\)\num{512}, \num{1024}, \num{2048}, \num{8192}\(\} \) and \( n \in \{\)\num{1024}, \num{1536}, \num{2048}, \num{3072}, \num{4096}, \num{8192}, \num{12288}, \num{16384}\(\} \), and some matrices are formed by concatenating multiple blocks of different types. The sparsity level (i.e., number of nonzeros in \( x \)) increases with \( m \), and for larger instances, the variation in support sizes is more pronounced.


Figure~\ref{fig:lorenzfull} shows the performance profile comparing \BPhocbin with {\tt ISAL1} and the LP formulation solved using {\tt HiGHS}. As one can see, \BPhocbin comes out as a clear winner against {\tt ISAL1}. It also surprisingly outperforms {\tt HiGHS}. However, in this case, we see a fast increase in the {\tt HiGHS} profile starting at $2^4$. This suggests that {\tt HiGHS} is performing well in a subgroup of the problems.

\begin{figure}[htbp]
\centering
\includegraphics[width=0.475\textwidth]{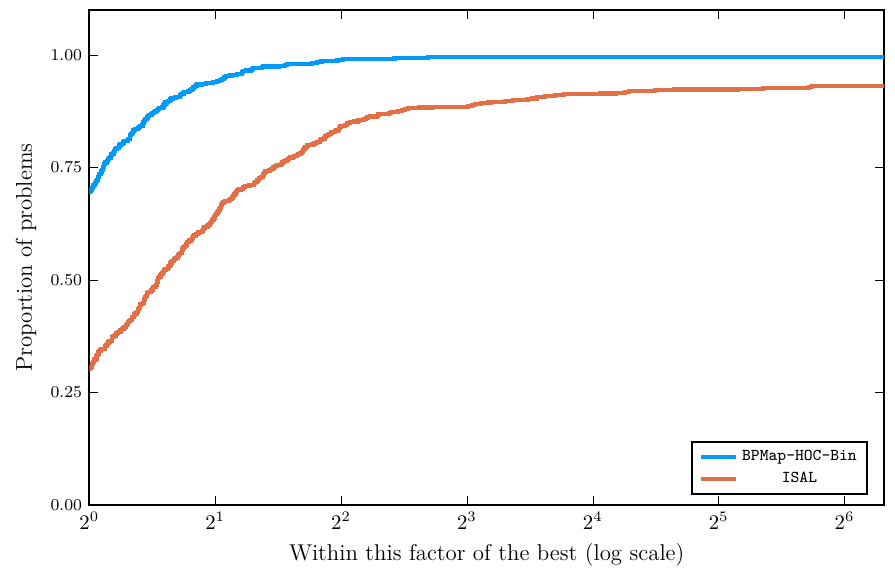}
\includegraphics[width=0.475\textwidth]{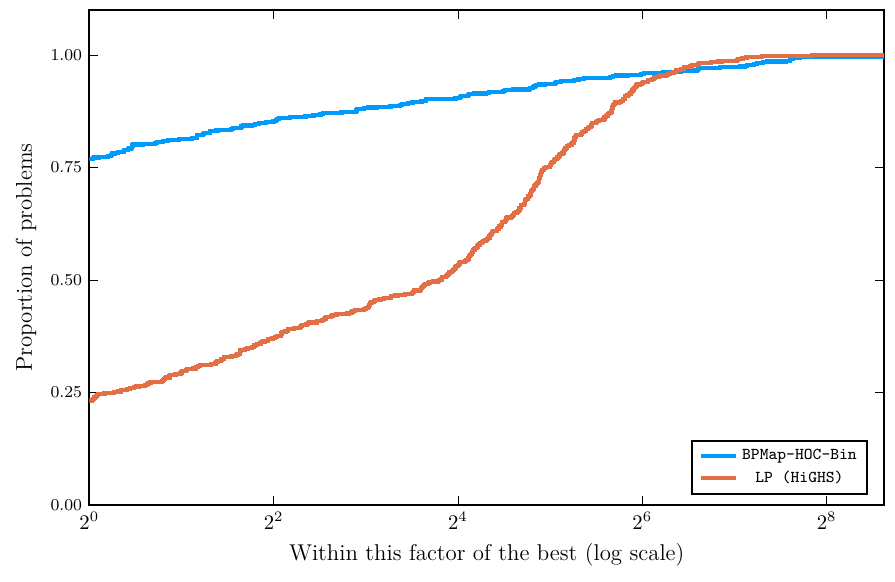}
\caption{Performance profile comparing \BPhocbin with {\tt ISAL1} and the LP
formulation solved with {\tt HiGHS} in full the Lorenz test set.}
\label{fig:lorenzfull}
\end{figure}

After some investigation, we have found that {\tt HiGHS} performs better for large problems, particularly when the matrices are sparse. We can see this in the performance profiles from Figure~\ref{fig:lorenzsizes}. In this case, we can confirm the good result when comparing \BPhocbin and {\tt ISAL1}, but there is a complete reversal between \BPhocbin and the LP model solved by {\tt HiGHS}. This result is somewhat in line with~\cite{Lorenz:2015}, where the LP implementations were also significantly faster for larger sizes. This suggests that the sparse linear algebra code is much more effective than the direct implementations in \BPhocbin and {\tt ISAL1}. In particular, HiGHS and other well-established LP codes use an extensive presolve phase that modifies the original problem formulation. In the context of \BPhocbin, this could probably be achieved by seeking suitable preconditioners for the linear systems. However, for iterative solvers like CG, the right preconditioner usually depends on each problem, as we did not find a simple one that works well across the whole test set. 

\begin{figure}[htbp]
\centering
\includegraphics[width=0.475\textwidth]{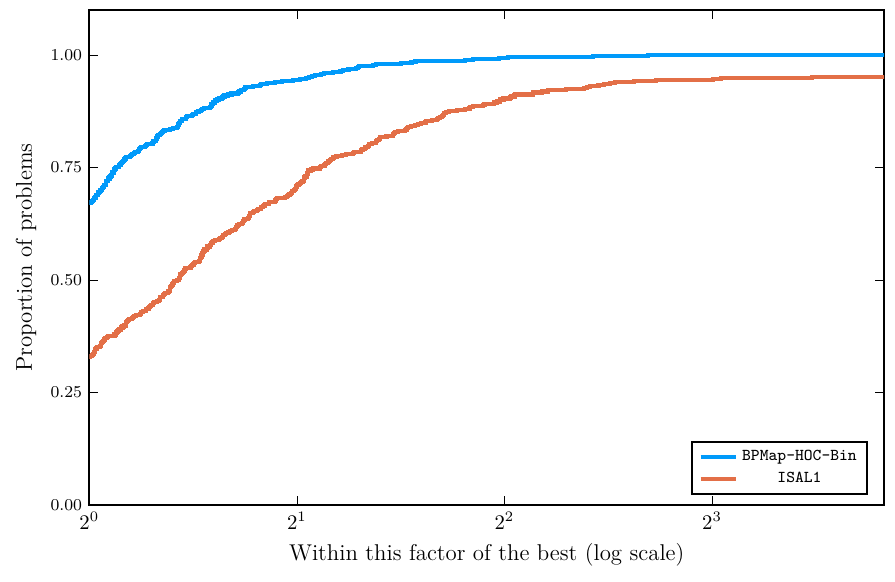}
\includegraphics[width=0.475\textwidth]{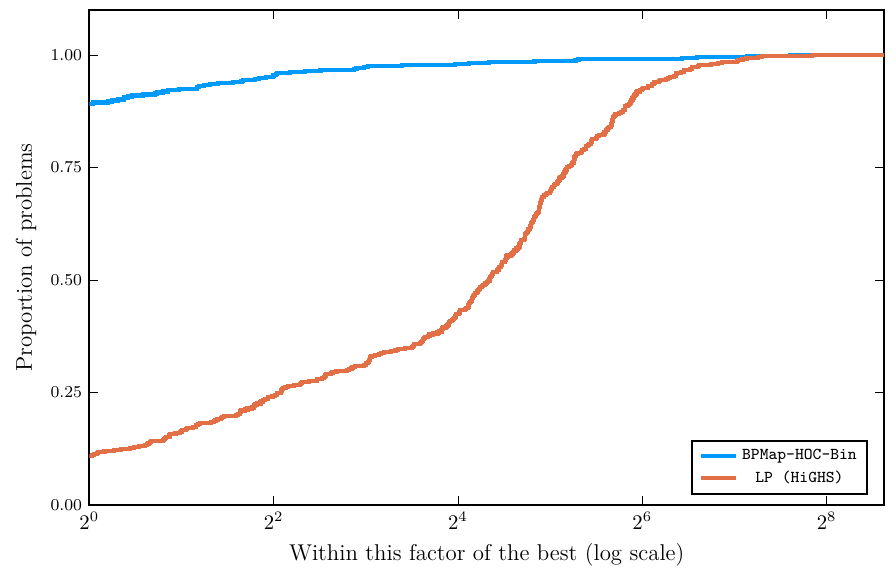} \\
\includegraphics[width=0.475\textwidth]{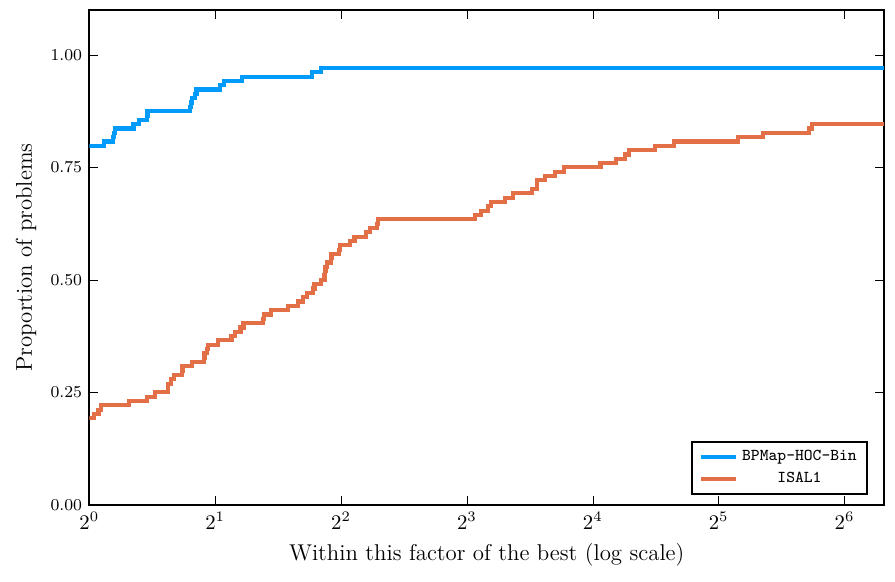}
\includegraphics[width=0.475\textwidth]{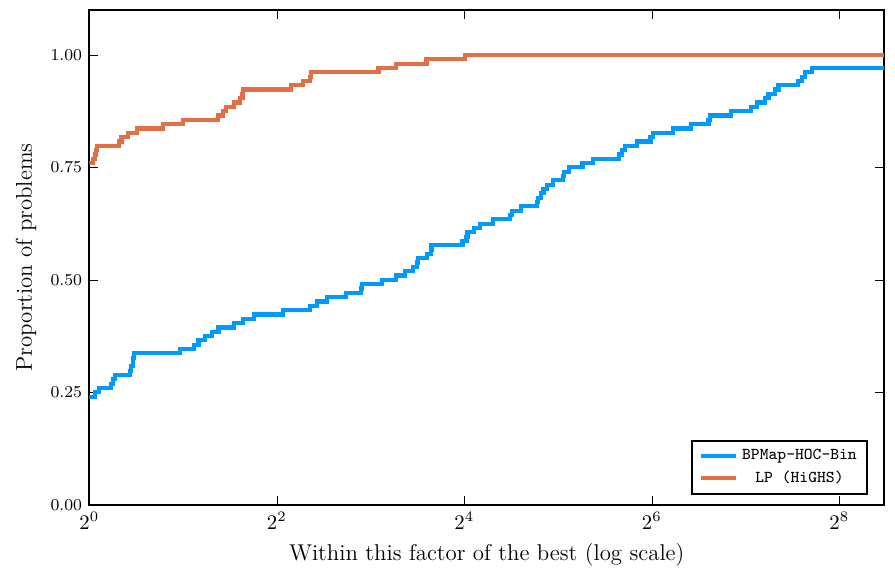} 
\caption{Performance profile comparing \BPhocbin with {\tt ISAL1} and the LP
formulation solved with {\tt HiGHS} the Lorenz test set. The top line show the
comparison for problems where at least one of the matrix dimensions is smaller
than $2048$. The second line is the results when both matrix dimensions are $2048$ or more.}
\label{fig:lorenzsizes}
\end{figure}

\subsubsection*{Real-world instances -- LSS test set}
\label{sec:realworld}

A second test set was based on a collection of problems compiled in~\cite{Lopes:2019}. This collection is particularly interesting because the data originate from real classification or regression applications, rather than being randomly generated. Since BP performs feature selection, we focused on problems that have more features (columns) than samples (lines), and hence we focused on the SC test set described in~\cite[Table 3]{Lopes:2019}. For each of these problems, we created four right-hand sides (rhs), vectors $b$, associated to solutions with sparsity levels $1\%,\ 5\%,\ 10\%$, and $20\%$. To do this, a full set of LASSO solutions for varying regularization parameters using the {\tt fit} function from {\tt LASSO.jl}~\cite{LASSOjl:2025} that implements the {\tt glmnet} algorithm~\cite{Friedman:2010}. The right-hand sides are then obtained by multiplying the data matrix $A$ by the solution that has the sparsity level closest to the target value. 

To avoid ill-conditioned problems, we have also attempted to solve the linear system $Ax = b_5$, where $b_5$ is the computed right-hand side with the sparsity level of approximately $5\%$. A problem is considered badly conditioned if one of the following conditions holds:
\begin{enumerate}
  \item If it is possible to factorize a dense version of matrix $A$ in memory and the residual of the solution obtained using the standard backslash Julia operator is greater than $\num{e-10}$.
  \item If $A$ is too large to be factorized and the {\tt cgne} function from {\tt Krylov.jl} function declares failure when trying to solve the system. 
\end{enumerate}
The remaining problems, used in all the tests below, are described in Table~\ref{tab:sctests}. All matrices are stored in sparse format.

\begin{table}[htbp]  \centering
    \caption{Problems used from LSS test set}
  \label{tab:sctests} 
  \begin{tabular}{llr}
    \toprule                                          
    Label & Problem name             & $(n_{row}, n_{col})$ \\ 
    \midrule
    SC1   & peppers05-6-6    & (32768, 65536)       \\
    SC2   & peppers05-12-12  & (32768, 65536)i      \\
    SC3   & peppers025-12-12 & (16384, 65536)       \\
    SC6   & SparcoProblem603 & (1024, 4096)         \\
    SC7   & finance1000      & (30465, 216842)      \\
    SC8   & dbworld-bodies   & (64, 4702)           \\
    SC9   & dexter-train     & (300, 7751)          \\
    SC10  & dexter-valid     & (300, 7847)          \\
    SC11  & dorothea-train   & (800, 88119)         \\
    SC12  & dorothea-valid   & (350, 72113)         \\
    SC13  & news20-binary    & (19996, 1355191)     \\
    SC14  & news20-scale     & (15935, 60346)       \\
    SC15  & news20-t-scale   & (3993, 39128)        \\
    SC17  & rcv1-train-mult  & (15564, 36842)       \\
    SC19  & sector-t-scale   & (3207, 39234)        \\
    SC21  & farm-ads-vect    & (4143, 54877)        \\
    SC22  & mug05-12-12      & (12410, 24820)       \\
    SC23  & mug025-12-12     & (6205, 24820)        \\
    SC24  & mug075-12-12     & (13651, 24820)       \\
    \bottomrule
  \end{tabular}
\end{table}

Figure~\ref{fig:LSSfull} shows the performance profiles comparing \BPhocbin with {\tt ISAL1} and the {\tt HiGHS} implementation. Again, we can see that \BPhocbin outperforms {\tt ISAL1}. But the comparison with {\tt HiGHS} is less clear. \BPhocbin starts slower, but, at the end, it shows to be more robust, solving more problems than {\tt HiGHS} in the one hour budget. 

\begin{figure}[htbp]
\centering
\includegraphics[width=0.475\textwidth]{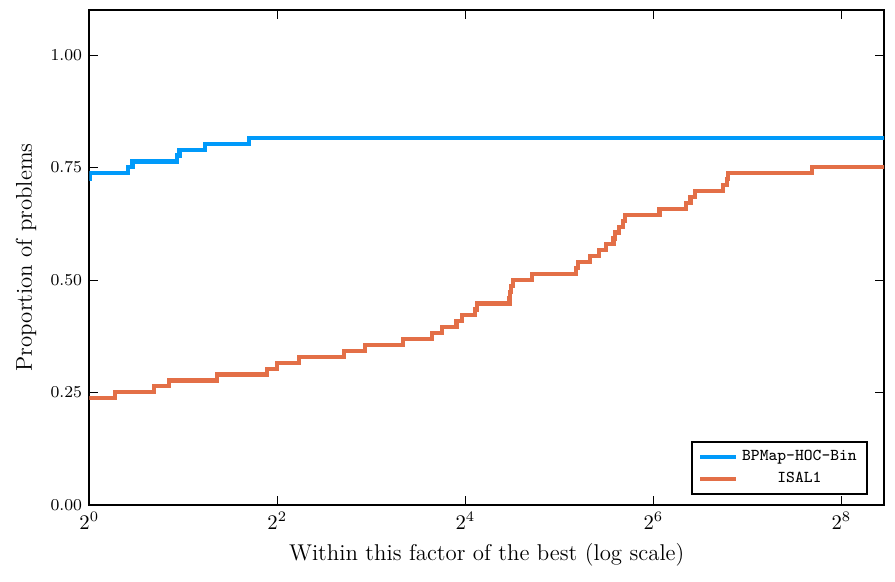}
\includegraphics[width=0.475\textwidth]{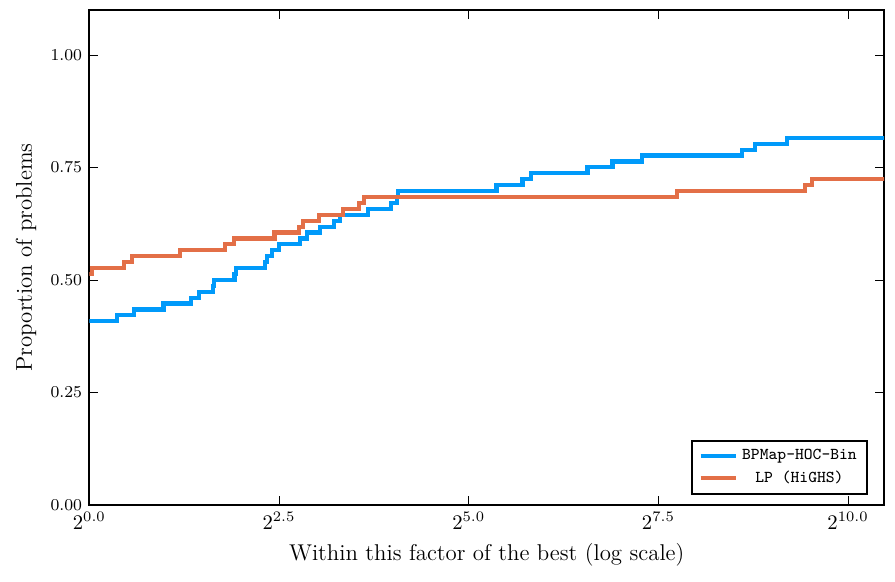}
\caption{Performance profile comparing \BPhocbin with {\tt ISAL1} and the LP
formulation solved with {\tt HiGHS} in full the LSS test set.}
\label{fig:LSSfull}
\end{figure}

When trying to understand these results, we observed that some of the LSS problems have a small number of rows. Hence, they are well-suited for a dual simplex approach, which is implemented in HiGHS. This is particularly relevant, as \BPhocbin is solving the linear system using Krylov iterative methods. We then decided to focus on problems with at least 1,000 rows and 10,000 columns, which can be considered medium to large in size. In this case, \BPhocbin can also benefit from moving the solution of the linear systems to a GPU to compensate for the advantages that {\tt HiGHS} showed on the sparse problems in the Lorenz test set. The computer node is equipped with 2 NVidia L40S GPUs. Only one GPU was used in the tests.  

Figure~\ref{fig:LSSlarge} shows the results. Now we see that \BPhocbin outperforms the solver using {\tt HiGHS}, in particular showing to be substantially more robust.  

\begin{figure}[htbp]
\centering
\includegraphics[width=0.475\textwidth]{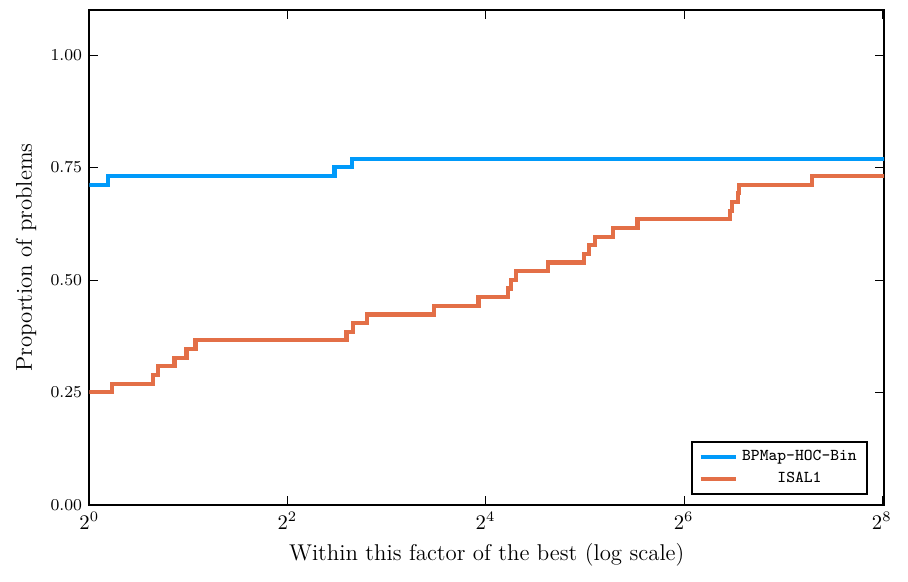}
\includegraphics[width=0.475\textwidth]{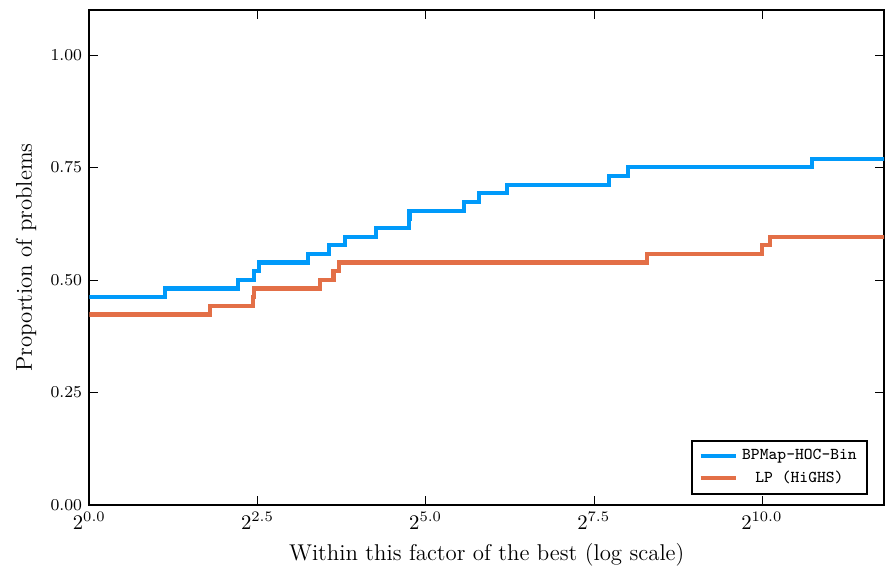}
\caption{Performance profile comparing \BPhocbin using GPU (NVidia L40S) with {\tt ISAL1} and the LP
formulation solved with {\tt HiGHS} in the problems from LSS test set with at least $1,000$ rows and $10,000$ columns.}
\label{fig:LSSlarge}
\end{figure}

\section{Concluding remarks}\label{sec:concluding}

In this paper, we have proposed \BPMAP, an algorithm for solving the Basis Pursuit problem \eqref{BP}. The method combines a projection-based outer loop with inner iterations via the Method of Alternating Projections (MAP), applied to a sequence of expanding $\ell_1$-balls. We proved that the radii $r_k$ converge linearly to the optimal value $\bar r$,
that all cluster points of $(x^k)_{k\in\NN}$ and $(z^k)_{k\in\NN}$ are solutions
of~\eqref{BP}, and that when \eqref{BP} has a unique solution $x^\ast$, both
sequences converge linearly to $x^\ast$ with rate $\rho = 1 - \omega/\sqrt{n}$,
where $\omega$ is the Hoffman constant of a natural error bound between $\Bbar$
and $\Mcal$. Practical variants \BPMAPhoc, incorporating a heuristic optimality check, and \BPMAPbin, using binary updates, further improved performance. While convergence is guaranteed in the case of a unique solution, the general case remains open.  Benchmark experiments on both synthetic and real-world instances confirmed the competitiveness of \BPMAP,  against state-of-the-art open-source solvers. The numerical evidence, together with the Fejér monotonicity properties of
MAP and the automatic error bound between $\Bbar$ and $\Mcal$, supports
convergence in the non-unique case, but a complete proof is left for future
work.

\bibliographystyle{siamplain}

\bibliography{refs}

@article{Bauschke:1993,
    title = {On the Convergence of von {{Neumann}}'s Alternating Projection Algorithm for Two Sets},
    author = {Bauschke, Heinz H and Borwein, Jonathan M},
    year = {1993},
    journal = {Set-Valued Analysis},
    volume = {1},
    number = {2},
    pages = {185--212},
    doi = {10.1007/BF01027691},
    langid = {english}
  }

@article{Bauschke:1994,
    title = {Dykstra's {{Alternating Projection Algorithm}} for {{Two Sets}}},
    author = {Bauschke, Heinz H and Borwein, Jonathan M},
    year = {1994},
    month = dec,
    journal = {Journal of Approximation Theory},
    volume = {79},
    number = {3},
    pages = {418--443},
    issn = {0021-9045},
    doi = {10.1006/jath.1994.1136},
    urldate = {2020-05-12},
    langid = {english}
  }

@article{Bauschke:1996,
    title = {On {{Projection Algorithms}} for {{Solving Convex Feasibility Problems}}},
    author = {Bauschke, Heinz H and Borwein, Jonathan M},
    year = {1996},
    journal = {SIAM Review},
    volume = {38},
    number = {3},
    pages = {367--426},
    doi = {10.1137/S0036144593251710},
    langid = {english}
  }

@book{Bauschke:2017a,
    title = {Convex {{Analysis}} and {{Monotone Operator Theory}} in {{Hilbert Spaces}}},
    author = {Bauschke, Heinz H. and Combettes, Patrick L.},
    year = {2017},
    series = {{{CMS Books}} in {{Mathematics}}},
    edition = {2},
    publisher = {Springer International Publishing},
    address = {Cham, Switzerland},
    doi = {10.1007/978-3-319-48311-5},
    urldate = {2021-01-05},
    isbn = {978-3-319-48310-8},
    langid = {english}
  }

@article{Behling:2021,
    title = {Infeasibility and Error Bound Imply Finite Convergence of Alternating Projections},
    author = {Behling, Roger and {Bello-Cruz}, Yunier and Santos, Luiz-Rafael},
    year = {2021},
    journal = {SIAM Journal on Optimization},
    volume = {31},
    number = {4},
    eprint = {2008.03354},
    pages = {2863--2892},
    doi = {10.1137/20M1358669},
    urldate = {2020-08-11},
    archiveprefix = {arXiv},
    copyright = {All rights reserved}
  }

@misc{Behling:2024d,
    title = {Fej{\'e}r* Monotonicity in Optimization Algorithms},
    author = {Behling, Roger and {Bello-Cruz}, Yunier and Iusem, Alfredo Noel and Ribeiro, Ademir Alves and Santos, Luiz-Rafael},
    year = {2024},
    month = oct,
    number = {arXiv:2410.08331},
    eprint = {2410.08331},
    publisher = {arXiv},
    doi = {10.48550/arXiv.2410.08331},
    urldate = {2024-10-14},
    archiveprefix = {arXiv},
    copyright = {All rights reserved},
    pubstate = {preprint}
  }

@article{Bello-Cruz:2022,
    title = {Quadratic {{Growth Conditions}} and {{Uniqueness}} of {{Optimal Solution}} to {{Lasso}}},
    author = {{Bello-Cruz}, Yunier and Li, Guoyin and Nghia, Tran Thai An},
    year = {2022},
    month = jul,
    journal = {Journal of Optimization Theory and Applications},
    volume = {194},
    number = {1},
    pages = {167--190},
    issn = {1573-2878},
    doi = {10.1007/s10957-022-02013-2},
    urldate = {2025-07-14},
    langid = {english}
  }

@article{Bezanson:2017,
    title = {Julia: {{A Fresh Approach}} to {{Numerical Computing}}},
    author = {Bezanson, Jeff and Edelman, Alan and Karpinski, Stefan and Shah, Viral B},
    year = {2017},
    month = feb,
    journal = {SIAM Review},
    volume = {59},
    number = {1},
    pages = {65--98},
    doi = {10.1137/141000671}
  }

@article{Bruckstein:2009,
    title = {From {{Sparse Solutions}} of {{Systems}} of {{Equations}} to {{Sparse Modeling}} of {{Signals}} and {{Images}}},
    author = {Bruckstein, Alfred M. and Donoho, David L. and Elad, Michael},
    year = {2009},
    month = feb,
    journal = {SIAM Review},
    volume = {51},
    number = {1},
    pages = {34--81},
    publisher = {{Society for Industrial and Applied Mathematics}},
    issn = {0036-1445},
    doi = {10.1137/060657704},
    urldate = {2023-11-18}
  }

@article{Candes:2005,
    title = {Decoding by Linear Programming},
    author = {Cand{\`e}s, Emmanuel J. and Tao, Terence},
    year = {2005},
    month = dec,
    journal = {IEEE Transactions on Information Theory},
    volume = {51},
    number = {12},
    pages = {4203--4215},
    issn = {1557-9654},
    doi = {10.1109/TIT.2005.858979},
    urldate = {2023-11-03}
  }

@article{Candes:2006,
    title = {Robust Uncertainty Principles: Exact Signal Reconstruction from Highly Incomplete Frequency Information},
    shorttitle = {Robust Uncertainty Principles},
    author = {Cand{\`e}s, Emmanuel J. and Romberg, Justin and Tao, Terence},
    year = {2006},
    month = feb,
    journal = {IEEE Transactions on Information Theory},
    volume = {52},
    number = {2},
    pages = {489--509},
    issn = {1557-9654},
    doi = {10.1109/TIT.2005.862083},
    urldate = {2023-11-18}
  }

@article{Candes:2008a,
    title = {The Restricted Isometry Property and Its Implications for Compressed Sensing},
    author = {Cand{\`e}s, Emmanuel J.},
    year = {2008},
    month = may,
    journal = {Comptes Rendus Mathematique},
    volume = {346},
    number = {9},
    pages = {589--592},
    issn = {1631-073X},
    doi = {10.1016/j.crma.2008.03.014},
    urldate = {2024-04-04}
  }

@article{Chen:2001,
    title = {Atomic {{Decomposition}} by {{Basis Pursuit}}},
    author = {Chen, Scott Shaobing and Donoho, David L. and Saunders, Michael A.},
    year = {2001},
    month = jan,
    journal = {SIAM Review},
    volume = {43},
    number = {1},
    pages = {129--159},
    issn = {0036-1445, 1095-7200},
    doi = {10.1137/S003614450037906X},
    urldate = {2023-04-19},
    langid = {english}
  }

@article{Cheney:1959,
    title = {Proximity {{Maps}} for {{Convex Sets}}},
    author = {Cheney, Ward and Goldstein, Allen A.},
    year = {1959},
    journal = {Proceedings of the American Mathematical Society},
    volume = {10},
    number = {3},
    eprint = {2032864},
    eprinttype = {jstor},
    pages = {448--450},
    publisher = {American Mathematical Society},
    issn = {0002-9939},
    doi = {10.2307/2032864},
    urldate = {2020-04-24}
  }

@article{Combettes:1990,
    title = {Method of Successive Projections for Finding a Common Point of Sets in Metric Spaces},
    author = {Combettes, P. L. and Trussell, H. J.},
    year = {1990},
    month = dec,
    journal = {Journal of Optimization Theory and Applications},
    volume = {67},
    number = {3},
    pages = {487--507},
    issn = {1573-2878},
    doi = {10.1007/BF00939646},
    urldate = {2020-10-20},
    langid = {english}
  }

@incollection{Combettes:2001a,
    title = {Quasi-{{Fej{\'e}rian Analysis}} of {{Some Optimization Algorithms}}},
    booktitle = {Studies in {{Computational Mathematics}}},
    author = {Combettes, Patrick L.},
    editor = {Butnariu, Dan and Censor, Yair and Reich, Simeon},
    year = {2001},
    month = jan,
    series = {Inherently {{Parallel Algorithms}} in {{Feasibility}} and {{Optimization}} and Their {{Applications}}},
    volume = {8},
    pages = {115--152},
    publisher = {Elsevier},
    doi = {10.1016/S1570-579X(01)80010-0},
    urldate = {2024-01-30}
  }

@article{Cominetti:2014,
    title = {A {{Newton}}'s Method for the Continuous Quadratic Knapsack Problem},
    author = {Cominetti, Roberto and Mascarenhas, Walter F. and Silva, Paulo J. S.},
    year = {2014},
    month = jun,
    journal = {Mathematical Programming Computation},
    volume = {6},
    number = {2},
    pages = {151--169},
    issn = {1867-2957},
    doi = {10.1007/s12532-014-0066-y},
    urldate = {2024-05-27},
    langid = {english}
  }

@article{Condat:2016,
    title = {Fast Projection onto the Simplex and the \${\textbackslash}ell\_\{1\}\$-Ball},
    author = {Condat, Laurent},
    year = {2016},
    month = jul,
    journal = {Mathematical Programming},
    volume = {158},
    number = {1},
    pages = {575--585},
    issn = {1436-4646},
    doi = {10.1007/s10107-015-0946-6},
    urldate = {2021-11-22},
    langid = {english}
  }

@article{Demanet:2016,
    title = {Eventual Linear Convergence of the {{Douglas-Rachford}} Iteration for Basis Pursuit},
    author = {Demanet, Laurent and Zhang, Xiangxiong},
    year = {2016},
    month = jan,
    journal = {Mathematics of Computation},
    volume = {85},
    number = {297},
    pages = {209--238},
    doi = {10.1090/mcom/2965},
    langid = {english}
  }

@article{Dessole:2023,
    title = {The {{Lawson-Hanson}} Algorithm with Deviation Maximization: {{Finite}} Convergence and Sparse Recovery},
    shorttitle = {The {{Lawson-Hanson}} Algorithm with Deviation Maximization},
    author = {Dessole, Monica and Dell'Orto, Marco and Marcuzzi, Fabio},
    year = {2023},
    journal = {Numerical Linear Algebra with Applications},
    volume = {30},
    number = {5},
    pages = {e2490},
    issn = {1099-1506},
    doi = {10.1002/nla.2490},
    urldate = {2024-04-04},
    copyright = {{\copyright} 2023 John Wiley \& Sons Ltd.},
    langid = {english}
  }

@article{Dolan:2002,
    title = {Benchmarking Optimization Software with Performance Profiles},
    author = {Dolan, Elisabeth D and Mor{\'e}, Jorge J},
    year = {2002},
    journal = {Mathematical Programming},
    volume = {91},
    number = {2},
    pages = {201--213},
    doi = {10.1007/s101070100263}
  }

@article{Donoho:2006,
    title = {Compressed Sensing},
    author = {Donoho, D.L.},
    year = {2006},
    month = apr,
    journal = {IEEE Transactions on Information Theory},
    volume = {52},
    number = {4},
    pages = {1289--1306},
    issn = {0018-9448},
    doi = {10.1109/TIT.2006.871582},
    urldate = {2023-11-18}
  }

@article{Gilbert:2017,
    title = {On the {{Solution Uniqueness Characterization}} in the {{L1 Norm}} and {{Polyhedral Gauge Recovery}}},
    author = {Gilbert, Jean Charles},
    year = {2017},
    month = jan,
    journal = {Journal of Optimization Theory and Applications},
    volume = {172},
    number = {1},
    pages = {70--101},
    issn = {1573-2878},
    doi = {10.1007/s10957-016-1004-0},
    urldate = {2025-07-14},
    langid = {english}
  }

@book{Golub:2013,
    title = {Matrix Computations},
    author = {Golub, Gene H. and Van Loan, Charles F.},
    year = {2013},
    series = {Johns {{Hopkins}} Studies in the Mathematical Sciences},
    edition = {4},
    publisher = {The Johns Hopkins University Press},
    address = {Baltimore},
    isbn = {978-1-4214-0794-4},
    lccn = {QA188 .G65 2013}
  }

@article{Hesse:2014,
    ids = {Hesse:2014a,Hesse:2014gi,Hesse:2014gia},
    title = {Alternating {{Projections}} and {{Douglas-Rachford}} for {{Sparse Affine Feasibility}}},
    author = {Hesse, Robert and Luke, D. Russell and Neumann, Patrick},
    year = {2014},
    month = sep,
    journal = {IEEE Transactions on Signal Processing},
    volume = {62},
    number = {18},
    pages = {4868--4881},
    issn = {1941-0476},
    doi = {10.1109/TSP.2014.2339801},
    uri = {[object Object]}
  }

@article{Hoffman:1952,
    title = {On Approximate Solutions of Systems of Linear Inequalities},
    author = {Hoffman, Alan J.},
    year = {1952},
    journal = {Journal of Research of the National Bureau of Standards},
    volume = {49},
    number = {4},
    pages = {263--265},
    doi = {10.6028/jres.049.027}
  }

@article{Huangfu:2018a,
    title = {Parallelizing the Dual Revised Simplex Method},
    author = {Huangfu, Q. and Hall, J. A. J.},
    year = {2018},
    month = mar,
    journal = {Mathematical Programming Computation},
    volume = {10},
    number = {1},
    pages = {119--142},
    issn = {1867-2949, 1867-2957},
    doi = {10.1007/s12532-017-0130-5},
    urldate = {2025-02-27},
    langid = {english}
  }

@article{Lopes:2019,
    title = {Accelerating Block Coordinate Descent Methods with Identification Strategies},
    author = {Lopes, R. and Santos, S. A. and Silva, P. J. S.},
    year = {2019},
    month = apr,
    journal = {Computational Optimization and Applications},
    volume = {72},
    number = {3},
    pages = {609--640},
    issn = {1573-2894},
    doi = {10.1007/s10589-018-00056-8},
    urldate = {2024-06-10},
    langid = {english}
  }

@article{Lorenz:2013,
    title = {Constructing {{Test Instances}} for {{Basis Pursuit Denoising}}},
    author = {Lorenz, Dirk A.},
    year = {2013},
    month = mar,
    journal = {IEEE Transactions on Signal Processing},
    volume = {61},
    number = {5},
    pages = {1210--1214},
    issn = {1053-587X, 1941-0476},
    doi = {10.1109/TSP.2012.2236322},
    urldate = {2023-02-10}
  }

@article{Lorenz:2014,
    title = {An Infeasible-Point Subgradient Method Using Adaptive Approximate Projections},
    author = {Lorenz, Dirk A. and Pfetsch, Marc E. and Tillmann, Andreas M.},
    year = {2014},
    month = mar,
    journal = {Computational Optimization and Applications},
    volume = {57},
    number = {2},
    pages = {271--306},
    issn = {1573-2894},
    doi = {10.1007/s10589-013-9602-3},
    urldate = {2023-11-04},
    langid = {english}
  }

@article{Lorenz:2015,
    title = {Solving {{Basis Pursuit}}: {{Heuristic Optimality Check}} and {{Solver Comparison}}},
    shorttitle = {Solving {{Basis Pursuit}}},
    author = {Lorenz, Dirk A. and Pfetsch, Marc E. and Tillmann, Andreas M.},
    year = {2015},
    month = feb,
    journal = {ACM Transactions on Mathematical Software},
    volume = {41},
    number = {2},
    pages = {8:1--8:29},
    issn = {0098-3500},
    doi = {10.1145/2689662},
    urldate = {2023-11-04}
  }

@article{Montoison:2023,
    title = {Krylov.Jl: {{A Julia}} Basket of Hand-Picked {{Krylov}} Methods},
    shorttitle = {Krylov.Jl},
    author = {Montoison, Alexis and Orban, Dominique},
    year = {2023},
    month = sep,
    journal = {Journal of Open Source Software},
    volume = {8},
    number = {89},
    pages = {5187},
    publisher = {The Open Journal},
    issn = {2475-9066},
    doi = {10.21105/joss.05187},
    urldate = {2025-07-14},
    copyright = {http://creativecommons.org/licenses/by/4.0/}
  }

@article{Mousavi:2019,
    title = {Solution Uniqueness of Convex Piecewise Affine Functions Based Optimization with Applications to Constrained {$\ell$}1 Minimization},
    author = {Mousavi, Seyedahmad and Shen, Jinglai},
    year = {2019},
    journal = {ESAIM: Control, Optimisation and Calculus of Variations},
    volume = {25},
    pages = {56},
    publisher = {EDP Sciences},
    issn = {1292-8119, 1262-3377},
    doi = {10.1051/cocv/2018061},
    urldate = {2025-07-14},
    copyright = {{\copyright} EDP Sciences, SMAI 2019},
    langid = {english}
  }

@misc{Stella:2025,
    title  = {{{JuliaFirstOrder}}/{{ProximalOperators}}.Jl},
    author = {Stella, Lorenzo and Antonello, Niccol{\`o}},
    doi = {10.5281/zenodo.6583304},
    urldate = {2022-09-15},
  }

@article{Tibshirani:1996,
    title = {Regression {{Shrinkage}} and {{Selection}} via the {{Lasso}}},
    author = {Tibshirani, Robert},
    year = {1996},
    journal = {Journal of the Royal Statistical Society. Series B (Methodological)},
    volume = {58},
    number = {1},
    eprint = {2346178},
    eprinttype = {jstor},
    pages = {267--288},
    publisher = {[Royal Statistical Society, Oxford University Press]},
    issn = {0035-9246},
    urldate = {2025-02-28}
  }

@article{Tillmann:2014,
    title = {The {{Computational Complexity}} of the {{Restricted Isometry Property}}, the {{Nullspace Property}}, and {{Related Concepts}} in {{Compressed Sensing}}},
    author = {Tillmann, Andreas M. and Pfetsch, Marc E.},
    year = {2014},
    month = feb,
    journal = {IEEE Transactions on Information Theory},
    volume = {60},
    number = {2},
    pages = {1248--1259},
    issn = {0018-9448, 1557-9654},
    doi = {10.1109/TIT.2013.2290112},
    urldate = {2024-04-04},
    copyright = {https://ieeexplore.ieee.org/Xplorehelp/downloads/license-information/IEEE.html}
  }

@article{Tropp:2004,
    title = {Greed Is {{Good}}: {{Algorithmic Results}} for {{Sparse Approximation}}},
    shorttitle = {Greed Is {{Good}}},
    author = {Tropp, J.A.},
    year = {2004},
    month = oct,
    journal = {IEEE Transactions on Information Theory},
    volume = {50},
    number = {10},
    pages = {2231--2242},
    issn = {0018-9448},
    doi = {10.1109/TIT.2004.834793},
    urldate = {2023-04-11},
    langid = {english}
  }

@article{Zhang:2015,
    title = {Necessary and {{Sufficient Conditions}} of {{Solution Uniqueness}} in 1-{{Norm Minimization}}},
    author = {Zhang, Hui and Yin, Wotao and Cheng, Lizhi},
    year = {2015},
    month = jan,
    journal = {Journal of Optimization Theory and Applications},
    volume = {164},
    number = {1},
    pages = {109--122},
    issn = {1573-2878},
    doi = {10.1007/s10957-014-0581-z},
    urldate = {2025-07-14},
    langid = {english}
  }

@misc{LASSOjl:2025,
  title = {{{JuliaStats}}/{{Lasso}}.Jl},
  urldate = {2025-06-06},
  origdate = {2015-01-04T21:56:08Z},
  organization = {Julia Statistics},
  url = {https://github.com/JuliaStats/Lasso.jl/}
}

@article{Friedman:2010,
  title = {Regularization {{Paths}} for {{Generalized Linear Models}} via {{Coordinate Descent}}},
  author = {Friedman, Jerome H. and Hastie, Trevor and Tibshirani, Rob},
  date = {2010-02-02},
  journal = {Journal of Statistical Software},
  volume = {33},
  pages = {1--22},
  langid = {english},
  year = {2010},
  doi = {10.18637/jss.v033.i01},
}

@misc{Secchin2026,
  title = {Parallel {{Newton}} Methods for the Continuous Quadratic Knapsack Problem: {{A Jacobi}} and {{Gauss-Seidel}} Tale},
  shorttitle = {Parallel {{Newton}} Methods for the Continuous Quadratic Knapsack Problem},
  author = {Secchin, Leonardo D. and Silva, Paulo J. S.},
  year = 2026,
  month = mar,
  number = {arXiv:2603.15910},
  eprint = {2603.15910},
  primaryclass = {math},
  publisher = {arXiv},
  archiveprefix = {arXiv}
}

\end{document}